\documentclass[a4paper,12pt]{article}

\usepackage[active]{srcltx}
\usepackage{hyperref}

\usepackage{amsmath,amsthm,amssymb,mathrsfs,latexsym,amsfonts}
\usepackage{graphicx,psfrag,epsfig}
\usepackage{bbm}
\usepackage[english]{babel}
\usepackage[latin1]{inputenc}
\usepackage{a4wide}

\newtheorem{theorem}{Theorem}

\newtheorem{proposition}[theorem]{Proposition}

\newtheorem{question}{Question}

\theoremstyle{definition}

\newtheorem{Main}{Theorem}

\def\N{\mathbb N}
\def\C{\mathbb C}

\def\R{\mathbb R}
\def\T{\mathbb T}

\def\Z{\mathbb Z}

\begin{document}

\begin{titlepage}
  \title{\LARGE{\textbf{Some remarks on the optimality of the Bruno-R{\"u}ssmann condition}}}%
  \author{Abed Bounemoura \\
    CNRS - PSL Research University\\
    (Universit{\'e} Paris-Dauphine and Observatoire de Paris)}
\end{titlepage}

\maketitle

\begin{abstract}
We prove that the Bruno-R{\"u}ssmann condition is optimal for the analytic preservation of a quasi-periodic invariant curve for an analytic twist map. The proof is based on Yoccoz's corresponding result for analytic circle diffeomorphisms and the uniqueness of invariant curves with a given irrational rotation number. We also prove a similar result for analytic Tonelli Hamiltonian flow with $n=2$ degrees of freedom; for $n \geq 3$ we only obtain a weaker result which recovers and slightly improves a theorem of Bessi.  
\end{abstract}

\section{Introduction}\label{s1}

Given $n \geq 2$, a vector $\omega \in \R^n$ satisfies the Bruno-R{\"u}ssmann condition, and we will write $\omega \in \mathrm{BR}$, if
\begin{equation}\label{BR}
\int_{1}^{+\infty}\frac{\ln(\Psi_{\omega}(Q))}{Q^2}dQ <
+\infty \tag{$\mathrm{BR}$}
\end{equation} 
where
\begin{equation*}
\Psi_{\omega}(Q)=\max\left\{|k\cdot\omega|^{-1}\; | \; k \in \Z^n, \; 0 < |k|\leq Q\right\}.
\end{equation*}
The expression in~\eqref{BR} is just one of the many equivalent ways of defining this Bruno-R{\"u}ssmann condition. Bruno (\cite{Bru71}, ~\cite{Bru72}, ~\cite{Bru89}) and R{\"u}ssmann (\cite{Rus80}, ~\cite{Rus89}, ~\cite{Rus94}, ~\cite{Rus01}) have proved that $\omega \in \mathrm{BR}$ is a sufficient condition for several analytic small divisors problems: among others, for the linearization of a holomorphic germ at a non-resonant fixed point, for the linearization of a torus diffeomorphism isotopic to the identity (respectively a torus vector field) close to a non-resonant translation (respectively close to a non-resonant constant vector field) and for the preservation of a non-resonant quasi-periodic invariant torus in a non-degenerate Hamiltonian system close to integrable.

For $n=2$, $\omega=(1,\alpha) \in \mathrm{BR}$ if and only if $\alpha$ satisfies the following Bruno condition, that we shall write $\alpha \in \mathrm{B}$: 
\begin{equation}\label{B}
\sum_{n \in \N}\frac{\log q_{n+1}}{q_n} < +\infty \tag{$\mathrm{B}$}
\end{equation}     
where $q_n$ is the denominator of the $n^{th}$-convergent of $\alpha$. It is a deep result of Yoccoz (\cite{Yoc88},~\cite{Yoc95}) that if $\alpha \notin \mathrm{B}$, then the quadratic polynomial
\[ P_\lambda(z)=\lambda z+z^2, \quad \lambda=e^{2\pi i \alpha} \]
is not analytically linearizable. Other examples of non-Bruno non-linearizable germs were later given by Geyer (\cite{Gey01}). Using this, Yoccoz was able to prove that if $\alpha \notin B$, there exist, arbitrarily close to the rotation $\alpha$, analytic circle diffeomorphisms which are topologically but not analytically conjugate to $\alpha$ and thus in the continuous case, if $\omega \notin \mathrm{BR}$, there exist, arbitrarily close to the constant vector field $\omega$, analytic vector fields on $\T^2$ which are topologically but not analytically conjugate to $\omega$ (see Theorems~\ref{thY1} and~\ref{thY2} below for more precise statements). The condition $\alpha \in \mathrm{B}$ (or equivalently $\omega \in \mathrm{BR}$) is also known to be optimal in other problems in $\C^2$, for vector fields close to a non-resonant singular point (\cite{PM97}) and for the complex area-preserving map known as the semi-standard map (\cite{Mar90}).

Unfortunately, to the best of our knowledge, the Bruno-R{\"u}ssmann condition is not known to be optimal for low-dimensional Hamiltonian problems such as the analytic preservation of invariant curves for twist maps. Here it is important to point out that unlike the other problems we mentioned which deal only with the existence of an analytic conjugacy to the linear model, in the Hamiltonian case the conclusions of KAM-like theorems are two-fold: it gives the existence of an analytic invariant curve together with the existence of an analytic conjugacy of the restricted dynamics on the curve to the linear model. The best known result for twist maps is due to Forni (\cite{For94})\footnote{Unfortunately, at several places in the literature (for instance~\cite{Gen15} and other references by the same author) it is stated that $\alpha \in \mathrm{B}$ is optimal for the existence of an analytic invariant circle for the standard map in the perturbative regime which depends analytically in the small parameter; we would like to point out that this statement is incorrectly deduced from results of Marmi (\cite{Mar90}) and Davie (\cite{Dav94}) and thus the optimality of $\alpha \in \mathrm{B}$ for the standard map is still an open question (see \cite{MM00} where this observation is also made).}. To describe his result, let us first remark that $\alpha \in \mathrm{B}$ obviously implies that $\alpha \in \mathrm{R}$ in the sense that
\begin{equation}\label{R}
\lim_{n \rightarrow +\infty} \frac{\log q_{n+1}}{q_n}=0 \tag{$\mathrm{R}$} 
\end{equation}
but clearly the converse is not true. The condition that $\alpha \in \mathrm{R}$ is in fact the necessary and sufficient condition for the linearized problem (the so-called cohomological equation) to have a solution in the analytic topology (\cite{Rus75}). Using results of Mather (\cite{Mat86}, \cite{Mat88}) and Herman (\cite{Her83}), Forni proved that if an integrable twist map has an invariant curve with rotation number $\alpha \notin \mathrm{R}$, then there exists arbitrarily small analytic perturbation for which there are no (necessarily Lipschitz) invariant curves with rotation number $\alpha$. Observe that this strongly violates the conclusion of the KAM theorem, as the latter would give an analytic invariant curve on which the dynamic is analytically linearizable. For Tonelli Hamiltonian flows with $n \geq 2$ degrees of freedom, a result analogous to Forni's has been obtained by Bessi (\cite{Bes00}). To state it, observe that a generalization of the condition $\alpha \in \mathrm{R}$ is (keeping the same notation) $\omega \in \mathrm{R}$ where
\begin{equation}\label{R}
\lim_{Q \rightarrow +\infty} \frac{\ln(\Psi_{\omega}(Q))}{Q}=0 \tag{$\mathrm{R}$} 
\end{equation}
and that again this is the necessary and sufficient condition to solve the cohomological equation in the analytic topology. Bessi proved that if $\omega \notin \mathrm{R}$, then there exists arbitrarily small perturbation of the integrable Hamiltonian $H_0(I)=\frac{1}{2}(I_1^2+\cdots+I_n^2)$ in the analytic topology for which there is no invariant $C^1$ Lagrangian graph on which the dynamic is $C^1$ conjugated to the linear flow of frequency $\omega$.

The purpose of this note is to prove that the condition $\alpha \in \mathrm{B}$ is optimal for the analytic KAM theorem for twist maps, in the sense that if $\alpha \notin \mathrm{B}$, then there exist arbitrarily small perturbations of an arbitrary integrable twist map for which there are no analytic invariant curves on which the dynamic is analytically conjugated to $\alpha$. We refer to Theorem~\ref{th1} for a more precise statement. One has to observe that this result does not improve Forni's result, as in our example, the perturbed map will have an analytic invariant curve on which the dynamic is topologically conjugated to $\alpha$, yet there will be no analytic conjugacy and this is sufficient to guarantee that the conclusions of the KAM theorem do not hold. One can considered Forni's result as a ``destruction" of invariant circle with rotation number $\alpha \notin \mathrm{R}$, while our result can be considered as a ``destruction" of the dynamic on the invariant circle with rotation number $\alpha \notin \mathrm{B}$. For perturbations of Tonelli Hamiltonians, we will obtain in Theorem~\ref{th2} a similar result showing the optimality of $\omega \in \mathrm{BR}$ for $n=2$ while for $n \geq 3$, we will only obtain in Theorem~\ref{th3} a result similar to Bessi showing that $\omega \in \mathrm{R}$ is necessary: for $n \geq 3$, it is unlikely that $\omega \in \mathrm{R}$ is sufficient and one should not expect $\omega \in \mathrm{BR}$ to be necessary either\footnote{Yoccoz, private communication.}. Even though we will use action-minimizing properties of invariant quasi-periodic curves and tori in an indirect way, our method of proof is very different from those of Forni and Bessi. For Theorem~\ref{th1}, we will use Yoccoz's result showing the necessity of $\alpha \in \mathrm{B}$ for the analytic linearization of circle diffeomorphims, and the well-known fact that an invariant curve for a twist map with a given irrational rotation number is unique. Under some more assumptions, this uniqueness property has been shown to be true for Tonelli Hamiltonians in any number of degrees of freedom by Fathi, Giualiani and Sorrentino (\cite{FGS09}). Using this and a continuous version of Yoccoz's result, we will obtain Theorem~\ref{th2} for the case $n=2$ and for $n \geq 3$, we will make use of a result of Fayad (\cite{Fay02}) on reparametrized linear flows to obtain Theorem~\ref{th3}.

\section{The case of a twist map}\label{s2}

It will be more convenient for us to represent exact area-preserving map of the annulus $\T \times \R$, where $\T=\R/\Z$, by a ``Hamiltonian" generating function defined on the universal cover $\R^2$ of $\T \times \R$ (unlike~\cite{For94} where a ``Lagrangian" generating function is used). Given a smooth function $h: \R^2 \rightarrow \R$, the map
\[ \bar{f}=\bar{f}_h : \R^2 \rightarrow \R^2\]
defined by
\begin{equation*}
\bar{f}(\theta,I)=(\Theta,\mathcal{I}) \Longleftrightarrow
\begin{cases}
\mathcal{I}=I-\partial_\theta h(\theta,\mathcal{I}),\\
\Theta=\theta+\partial_{\mathcal{I}} h(\theta,\mathcal{I})
\end{cases}
\end{equation*}
projects to an exact area-preserving map
\[ f: \T \times \R \rightarrow \T \times \R. \] 
Such a map is an exact area-preserving twist map, or for short twist map in the sequel, if it satisfies the following two conditions:
\begin{itemize}
\item[(a1)] for all $(\theta,I) \in \R^2$, $\partial_I \Theta(\theta,I)>0$;
\item[(a2)] for all $\theta \in \R$, $|\Theta| \rightarrow +\infty$ as $|I| \rightarrow +\infty$ uniformly in $\theta$.
\end{itemize}  
Given such a twist map $f$, an invariant curve $T$ for $f$ will be an essential topological circle such that $f(T)=T$; necessarily, $T$ is a Lipschitz Lagrangian graph. Since $f$ preserves orientation, the restriction $f_{|T}$ has a well-defined rotation number. The following uniqueness result is well-known (see~\cite{KH95} for instance). 
 
\begin{proposition}\label{unique1}
Let $T_0$ and $T_1$ be two invariant curves for a twist map such that $f_{|T_0}$ and $f_{|T_1}$ have the same irrational rotation number. Then $T_0=T_1$.
\end{proposition}

Now let us explain the local setting in which the KAM theorem applies. Consider a smooth function $h_0: (-1,1) \rightarrow \R$ which satisfies the following conditions:
\begin{itemize}
\item[(b1)] $h_0''(\mathcal{I}) \neq 0$ for all $\mathcal{I} \in (-1,1)$;
\item[(b2)] $h_0'(0)=\alpha$ .
\end{itemize}
Then the exact area-preserving map $f_0$ generated by $h_0$ is integrable, and the dynamic restricted to the invariant curve $T_0=\T \times \{0\}$ is the rotation by $\alpha$. To state the KAM theorem of Bruno and R{\"u}ssmann, we need to define norms for real-analytic functions. Let $h : \R \times (-1,1) \rightarrow \R$ be a real-analytic function and suppose it admits a holomorphic and bounded extension (still denoted by $h$) to the domain
\[ \T_s \times D=\{z=(z_1,z_2) \in \C^2 \; | \; |\mathrm{Im} \; z_1|<s, \; |z_2|<1\} \]
for some $s>0$. In such a case, we simply define
\[ |h|_s=\sup_{z \in \T_s \times D}|h(z)|. \]
Assume that $h_0$ satisfies condition (b1) and (b2) with $\alpha \in \mathrm{B}$, then the KAM theorem states that for any $s>0$, there exists $\varepsilon>0$ such that for any $h_1$ satisfying $|h_1-h_0|_s \leq \varepsilon$, the exact area-preserving map $f_1$ generated by $h_1$ has an analytic invariant curve $T_1$ such that $f_{1|T_1}$ is analytically conjugate to the rotation $\alpha$ (and moreover, $T_1$ analytically converges to $T_0$ as $\varepsilon$ goes to zero). 

The following result shows that the condition that $\alpha \in \mathrm{B}$ cannot be weakened.

\begin{Main}\label{th1}
Assume that $h_0$ satisfies condition (b1) and (b2) with $\alpha \notin \mathrm{B}$. Then for all $\varepsilon>0$ sufficiently small and all $s>0$, there exists $h_1$ such that $|h_1-h_0|_s \leq \varepsilon$ and the exact area-preserving map $f_1$ generated by $h_1$ has no analytic invariant curve $T_1$ such that $f_{1|T_1}$ is analytically conjugate to the rotation $\alpha$.
\end{Main}

The restriction on $\varepsilon$ only comes from the condition (a1) and is thus independent of the choice of $s$. The proof of Theorem~\ref{th1} will follow easily from the following theorem of Yoccoz.

\begin{theorem}[Yoccoz]\label{thY1}
Assume $\alpha \notin \mathrm{B}$. Then for all $\varepsilon>0$ and all $s>0$, there exists an orientation-preserving analytic circle diffeomorphism with a lift of the form
\[ u(\theta)=\theta+\alpha+v(\theta), \quad |v|_s \leq \varepsilon \]
which is topologically but not analytically conjugate to the rotation $\alpha$.
\end{theorem}

\begin{proof}[Proof of Theorem~\ref{th1}]
Let us fix $s>0$ and $\varepsilon>0$, and consider the function $v : \R \rightarrow \R$ given by Theorem~\ref{thY1} which extends to $\T_s$ and satisfy $|v|_s \leq \varepsilon$. We set
\[ h_1(\theta,\mathcal{I})=h_0(\mathcal{I})+v(\theta)\mathcal{I} \]
so that obviously
\[ |h_1-h_0|_s \leq |v|_s \leq \varepsilon. \]
Let $f_0$ and $f_1$ be the maps generated by respectively $h_0$ and $h_1$. Obviously, the condition (b1) implies that condition (a1) is satisfied by $f_0$ but only for all $(\theta,I) \in \R \times (-1,1)$ and assuming $\varepsilon$ sufficiently small, the same remains true for $f_1$. Now using a bump function, the map $f_1$, initially defined on $\T \times (-1,1)$, can be extend to a smooth twist map from $\T \times \R$ to itself in such a way that both (a1) and (a2) holds true.

Now for all $(\theta,I) \in \R \times (-1,1)$, the lift $\bar{f}_1$ is defined by
\[ \bar{f}_1(\theta,I)=(\Theta,\mathcal{I}) \]
where
\begin{equation}\label{eq1}
\begin{cases}
\mathcal{I}=I-\partial_\theta h_1(\theta,\mathcal{I})=I-v'(\theta)\mathcal{I},\\
\Theta=\theta+\partial_{\mathcal{I}} h_1(\theta,\mathcal{I})=\theta+h_0'(\mathcal{I})+v(\theta).
\end{cases} 
\end{equation}
Since $u(\theta)=\theta+\alpha+v(\theta)$ is the lift of an orientation-preserving diffeomorphism of the circle, we have
\[ u'(\theta)=1+v'(\theta)>0 \]
and thus~\eqref{eq1} can be written again as
\begin{equation}\label{eq2}
\begin{cases}
\mathcal{I}=(u'(\theta))^{-1}I,\\
\Theta=\theta+h_0'\left((u'(\theta))^{-1}I\right)+v(\theta).
\end{cases} 
\end{equation}
From~\eqref{eq2} it is now clear that $T_0=\T \times \{0\}$ is invariant by $f_1$, and since $h'(0)=\alpha$ by (b2), the restriction $f_{1|T_0}$ is nothing but the dynamic induced by $u$ given by Theorem~\ref{thY1}, hence it is topologically but not analytically conjugate to the rotation $\alpha$. 

To conclude, we argue by contradiction and assume the existence of an analytic invariant curve $T_1$ such that $f_{1|T_1}$ is analytically conjugate to the rotation $\alpha$. Since both $T_0$ and $T_1$ are invariant by the twist map $f_1$ and have the same irrational rotation number $\alpha$, it follows from Proposition~\ref{unique1} that $T_0=T_1$ but then $f_{1|T_0}=f_{1|T_1}$ is analytically conjugate to the rotation $\alpha$, which is absurd.   
\end{proof}

\section{The case of a Hamiltonian flow}\label{s3}

By a suspension argument (see for instance~\cite{KP94} or~\cite{TreBook} for the analytic case), Theorem~\ref{th1} gives a result for Hamiltonian systems with $n=1,5$ degrees of freedom with a convex (non-degenerate) integrable part, and thus also for Hamiltonian systems with $n=2$ degrees of freedom with a quasi-convex (iso-energetically non-degenerate) integrable part.

For Hamiltonian systems with $n \geq 2$ degrees of freedom and a convex integrable part, to use the argument in the proof of Theorem~\ref{th1} one first needs to have an analog of Proposition~\ref{unique1}, and fortunately, such a result was proved in~\cite{FGS09}. The setting is the one of Tonelli Hamiltonians, which is a natural generalization of exact area-preserving twist maps. For more details on Tonelli Hamiltonians and what we will describe next, we refer to~\cite{Sor15}.

Let $H : \T^n \times \R^n \rightarrow \R$ be a smooth Hamiltonian, then it is said to be Tonelli if it satisfies the following two conditions
\begin{itemize}
\item[(A1)] for all $(\theta,I) \in \T^n \times \R^n$, $\nabla_I^2 H(\theta,I)$ is a (uniformly) positive definite quadratic form;
\item[(A2)] for all $\theta \in \T^n$, one has 
\[ \lim_{|I| \rightarrow +\infty} \frac{H(\theta,I)}{|I|}=+\infty. \]
\end{itemize}  
In this context, the role of invariant curves is played by Lipschitz Lagrangian graphs, so let $T$ be such a graph, and assume it is invariant be the flow of a Tonelli Hamiltonian $H$. Given a measure supported on $T$ and invariant by the Hamiltonian flow, one can define a rotation vector (or a Schwartzman asymptotic cycle) as an element of $H_1(\T^n,\R)\simeq \R^n$ and by considering all invariant measures, one can define a rotation set for the Hamiltonian flow restricted to $T$. We will say that $T$ is Schwartzman strictly ergodic with rotation vector $\omega \in \R^n$ if its rotation set reduces to $\omega$ and there exists at least one measure with full support in $T$. Simple examples (the one we will actually use later) are when the restricted flow on $T$ is either topologically conjugate to the linear flow of frequency $\omega$ or obtained form the latter by a smooth reparametrization. The main result of~\cite{FGS09} gives the following statement.

\begin{theorem}[Fathi-Giuliani-Sorrentino]\label{unique2}
Let $T_0$ and $T_1$ be two Lipschitz Lagrangian graphs invariant by the flow of a Tonelli Hamiltonian and which are Schwartzman strictly ergodic with the same rotation vector $\omega \in \R^n$. Then $T_0=T_1$.
\end{theorem}

As before, we now come back to a local setting in which the KAM theorem applies. Consider a smooth Hamiltonian function $H_0: B \rightarrow \R$, where $B=(-1,1)^n \subseteq  \R^n$ is a unit ball and assume it satisfies the following conditions:
\begin{itemize}
\item[(B1)] $\nabla^2 H_0(I)$ is a positive definite quadratic form for all $I \in B$;
\item[(B2)] $\nabla H_0 (0)=\omega$ .
\end{itemize}
Let us observe that for the KAM theorem, condition (B1) is not needed as the weaker assumption that $\nabla^2 H_0(I)$ is a non-degenerate quadratic form is sufficient. If $H_0$ satisfies (B1) and (B2), its vector field $X^{H_0}$ is integrable and its restriction to the invariant torus $T_0=\T^n \times \{0\}$ is given by the constant vector field $\omega$.  

Let $H : \T^n \times B \rightarrow \R$ be a real-analytic function and suppose it admits a holomorphic and bounded extension to the domain
\[ \T^n_s \times D=\{z=(z_1,\dots,z_{2n}) \in \C^{n}/\Z^n \times \C^n \; | \; \max_{1 \leq i \leq n}|\mathrm{Im} \; z_i|<s, \; \max_{1 \leq i \leq n}|z_{n+i}|<1\} \]
for some $s>0$, so that we can define
\[ |H|_s=\sup_{z \in \T^n_s \times D}|H(z)|. \]
Here's a formulation of the KAM theorem for Tonelli Hamiltonians. Assume $H_0$ satisfies condition (B1) and (B2) with $\omega \in \mathrm{BR}$, then for any $s>0$, there exists $\varepsilon>0$ such that for any $H_1$ with $|H_1-H_0|_s \leq \varepsilon$, the Hamiltonian flow of $H_1$ has an analytic Lagrangian invariant torus $T_1$ which is a graph and such that the restriction $X^{H_1}_{|T_1}$ is analytically conjugate to the vector field $\omega$ (and moreover, $T_1$ analytically converges to $T_0$ as $\varepsilon$ goes to zero). 

For $n=2$, we can prove that the condition that $\omega \in \mathrm{BR}$ cannot be weakened. 

\begin{Main}\label{th2}
Let $n=2$, and assume that $H_0$ satisfies condition (B1) and (B2) with $\alpha \notin \mathrm{BR}$. Then for all $\varepsilon>0$ sufficiently small and all $s>0$, there exists $H_1$ such that $|H_1-H_0|_s \leq \varepsilon$ and the Hamiltonian flow of $H_1$ has no analytic Lagrangian invariant graph $T_1$ such that $X^{H_1}_{|T_1}$ is analytically conjugate to the vector field $\omega$.
\end{Main}

To prove Theorem~\ref{th2}, we will need the follwoing continuous version of Theorem~\ref{thY1}.

\begin{theorem}[Yoccoz]\label{thY2}
Assume $\omega \notin \mathrm{BR}$. Then for all $\varepsilon>0$ and all $s>0$, there exists an analytic vector field on $\T^2$ of the form
\[ U(\theta)=\omega+V(\theta), \quad |V|_s \leq \varepsilon \]
which is topologically but not analytically conjugate to vector field $\omega$.
\end{theorem}

\begin{proof}[Proof of Theorem~\ref{th2}]
The proof is just a continuous version of the proof of Theorem~\ref{th1}. Fixing $s>0$ and $\varepsilon>0$, we define
\[ H_1(\theta,I)=H_0(I)+V(\theta)\cdot I\]
where $V$ is given by Theorem~\ref{thY2} (for the value $\varepsilon/\sqrt{2}$ instead of $\varepsilon$). Clearly
\[ |H_1-H_0|_s \leq \sqrt{2}|V|_s \leq \varepsilon. \]
Observe that the condition (B1) and a smallness assumption on $\varepsilon$ allow again to extend $H_1$ to a smooth function defined on $\T^2 \times \R^2$ which satisfies both (A1) and (A2).

It is clear from the Hamiltonian's equation that $T_0=\T^2 \times \{0\}$ is invariant by the flow of $H_1$, and since $\nabla H_0(0)=\omega$ by (B2), the restriction $X^{H_1}_{|T_0}$ is nothing but the vector field $U$ given by Theorem~\ref{thY2}, hence it is topologically but not analytically conjugate to the vector field $\omega$. 

To conclude, we argue by contradiction and assume the existence of an analytic Lagrangian invariant graph $T_1$ such that $X^{H_1}_{|T_1}$ is analytically conjugate to the vector field $\omega$. As $T_0$ and $T_1$ are invariant by the Hamiltonian flow of $H_1$ which is Tonelli and are both Schwartzman strictly ergodic with the same rotation vector $\omega$, it follows from Theorem~\ref{unique2} that $T_0=T_1$. But then $X^{H_1}_{|T_0}=X^{H_1}_{|T_1}$ is analytically conjugate to the vector field $\omega$, which is absurd.   
\end{proof}

Theorem~\ref{thY2} is not known (and unlikely to be true) for $n \geq 3$, yet the following result was proved by Fayad in~\cite{Fay02}. 

\begin{theorem}[Fayad]\label{thF}
Let $n \geq 2$ and assume $\omega \notin \mathrm{R}$. Then for all $s>0$ sufficiently small and all $\varepsilon>0$, there exists an analytic vector field on $\T^n$ of the form
\[ U(\theta)=\omega+\varphi(\theta)\omega, \quad |\varphi|_s \leq \varepsilon, \quad \int_{\T^n}\varphi(\theta)d\theta=0, \]
which is not topologically conjugate to vector field $\omega$.
\end{theorem}

The restriction on $s$ is as follows: if $\omega \notin \mathrm{R}$, then there exists $s_0>0$ such that 
\[ \limsup_{Q \rightarrow +\infty} \frac{\ln(\Psi_{\omega}(Q))}{Q}\geq s_0 \]
and one has to choose $s<s_0$. We have to point out that Fayad's result is in fact much more general than the one we stated (it is not perturbative, valid for a $G^\delta$ dense set of functions $\varphi$ and the resulting vector field $U$ is in fact weakly mixing) but we will only use the above statement. Observe that since the flow of $U$ is a reparametrization (with a function of unit average) of the linear flow of frequency $\omega$, it is Schwartzman strictly ergodic with rotation vector $\omega$. 

Replacing Theorem~\ref{thY2} by Theorem~\ref{thF} in the proof of Theorem~\ref{th2}, one immediately arrives at the following statement.

\begin{Main}\label{th3}
Let $n\geq 2$, and assume that $H_0$ satisfies condition (B1) and (B2) with $\omega \notin \mathrm{R}$. Then for all $\varepsilon>0$ sufficiently small and all $s>0$ sufficiently small, there exists $H_1$ such that $|H_1-H_0|_s \leq \varepsilon$ and the Hamiltonian flow of $H_1$ has no Lipschitz Lagrangian invariant graph $T_1$ such that $X^{H_1}_{|T_1}$ is topologically conjugate to the vector field $\omega$.
\end{Main}

As we already explained, this statement is similar to the main result of~\cite{Bes00}. Yet Bessi's result depends on the choice of $H_0(I)=\frac{1}{2}(I_1^2+\cdots+I_n^2)$ while we can deal with an arbitrary integrable Hamiltonian $H_0$ which is convex in a neighborhood of the origin. Also as it is stated, the main result of~\cite{Bes00} claims the non-existence of a $C^1$ Lagrangian invariant graph $T_1$ such that $X^{H_1}_{|T_1}$ is $C^1$-conjugate to the vector field $\omega$ and thus our conclusion is slightly stronger; yet it seems to us that what is really proved in~\cite{Bes00} is the non-existence of a Lipschitz Lagrangian invariant graph $T_1$ such that $X^{H_1}_{|T_1}$ has all orbits with the same rotation vector $\omega$, in which case our conclusion could be slightly weaker. 

\section{Some questions}\label{s4}

Let us conclude by some questions. It is clear from Forni's result, Bessi's result or Theorem~\ref{th3} that when $\omega \notin \mathrm{R}$, invariant torus with a frequency $\omega$ are destroyed in a rather strong sense. But in Theorem~\ref{th1} and Theorem~\ref{th2} this is not the case if $\omega \notin \mathrm{BR}$ as an invariant analytic torus still exist on which the dynamic is topologically linearizable. So one may ask the following question.

\begin{question}
Assume that $\omega \in \mathrm{R} \setminus \mathrm{BR}$, is it possible to have the existence of a ``regular" invariant Lagrangian torus on which the conjugacy to the linear model is ``less regular"?
\end{question}

We have used quotation marks since we have no idea of what can be expected, the question is basically whether is it possible to prove anything non-trivial under the sole assumption that $\omega \in \mathrm{R}$, which as we already explained, is the condition that guarantees that the cohomological equation can be solved with an arbitrarily small loss of analyticity. Of course, it may well be the case that when $\omega \notin \mathrm{BR}$, the conclusions of Theorem~\ref{th1} and Theorem~\ref{th2} can be strengthened to reach conclusions similar to Forni and Bessi's results.

A second question concerns the assumptions (a1) and (A1). Clearly, (a1) is not a restriction as it is the natural non-degeneracy assumption under which an invariant curve with a prescribed frequency persists. But this is not the case for (A1) as we already pointed out, so we may ask the following question.

\begin{question}
Is it possible to prove Theorem~\ref{th2} and Theorem~\ref{th3} replacing the condition (A1) by the weaker condition that $\nabla_I^2 H_0$ is non-degenerate in a neighborhood of $0$? 
\end{question}

We expect the answer to be yes, at the expense of restricting the conclusion of non-existence to a neighborhood of the unperturbed torus. The role of the condition (A1) is to be able to obtain global uniqueness of invariant torus with a prescribed frequency; without (A1) no such global uniqueness has to be expected yet in view of the statement of the KAM theorem, only local uniqueness would be required. This local uniqueness is known to hold true within the context of KAM theory (see~\cite{Sal04} for instance) but this is not directly applicable to our context, yet we believe that with extra work this can be reached even though we did not pursue this further.

\medskip

\textit{Acknowledgements.} This work was done while the
author was in Cuba, in particular in the very nice caf\'{e} ``Tu t\'{e}" in Santa Clara. The author have also benefited
from partial funding from the ANR project Beyond KAM. 

\addcontentsline{toc}{section}{References}
\bibliographystyle{amsalpha}
\bibliography{BRcondition}

\end{document}